\documentclass{amsart}

\usepackage{amsmath, amssymb, amsthm}
\usepackage[mathcal]{eucal}
\usepackage{enumerate}

\newtheorem{theorem}{Theorem}
\newtheorem{prop}{Proposition}[section]

\newtheorem{lemma}{Lemma}[section]

\numberwithin{equation}{section}
\newtheorem*{Question}{Question}

\def\supp{{\rm supp\,}}

\def\dist{{\rm dist}}

\def\sgn{{\rm sgn}}

\def\ZQ{\ensuremath{\mathcal Q}}

\def\ZZ{\ensuremath{\mathbb Z}}

\def\ZR{\ensuremath{\mathbb R}}
\def\ZT{\ensuremath{\mathbb T}}

\def\Z1{\ensuremath{\mathbf 1}}

\newcommand {\e }[1]{(\ref{#1})}
\newcommand {\lem }[1]{Lemma \ref{#1}}

\newcommand {\trm }[1]{Theorem \ref{#1}}

\newcommand {\sect }[1]{Section \ref{#1}}
\newcommand {\subsect }[1]{Subsection \ref{#1}}

\newcommand\numberthis{\addtocounter{equation}{1}\tag{\theequation}}

\author{Gevorg Mnatsakanyan}
\title[on convergence of MT series]{On almost-everywhere convergence of Malmquist-Takenaka Series}
\address{Mathematical Institute, University of Bonn, Endenicher Allee 60, 53115 Bonn, Germany}
\curraddr{}
\email{gevorg@math.uni-bonn.de}
\subjclass[2020]{Primary 42A20, 42A50}
\keywords{Takenaka Basis, Phase unwinding, Time-Frequency Analysis, Carleson's Theorem}

\usepackage[dvipsinames]{xcolor}
\definecolor{midnightblue}{HTML}{0059b3}
\definecolor{chromered}{HTML}{f14233}
 \RequirePackage[colorlinks,citecolor=midnightblue,urlcolor=midnightblue,breaklinks=true,
linkbordercolor=midnightblue,linkcolor=midnightblue]{hyperref}

\begin{document}

\maketitle

\begin{abstract}
The Malmquist-Takenaka system is a perturbation of the classical trigonometric system, where powers of $z$ are replaced by products of other M\"obius transforms of the disc. The system is also inherently connected to the so-called nonlinear phase unwinding decomposition which has been in the center of some recent activity.
We prove $L^p$ bounds for the maximal partial sum operator of the Malmquist-Takenaka series under additional assumptions on the zeros of the M\"obius transforms. We locate the problem in the time-frequency setting and, in particular, we connect it to the polynomial Carleson theorem.
\end{abstract}

\section{Introduction}

Let $F$ be a holomorphic function in a neighborhood of the closed unit disc. We consider the holomorphic function $F_1$ defined on the same neighborhood by
$$F(z) = F(0) + z F_1(z).$$ Iterating this process, one obtains the Fourier series on the unit circle. We can consider a more general iterative process where we replace the factor $z$ by a different M\"obius transform of the unit disc vanishing at point $a_n$,
$$F_n(z) = F_{n} ( a_n ) + \frac{z- a_n}{1 - \overline{a_n}z} F_{n+1} (z).$$

Consider a sequence of points $(a_n)_{n=1}^{\infty}$ inside the unit disk such that
\begin{equation}\label{basiscondition}
\sum_{n=1}^{\infty} ( 1 - |a_n| ) = +\infty.
\end{equation}
The associated Blaschke products and Malmquist-Takenaka (MT) basis are defined as
\begin{equation}\label{BlaschkeMT}
B_n (z) = \prod_{j=1}^n \frac{z - a_j}{1 - \overline{a_j} z }, \quad\quad \phi_n (z) = B_{n} (z) \frac{ \sqrt{1 - | a_{n+1} |^2 } }{ 1 - \overline{ a_{n+1} } z }.
\end{equation}
Unlike the case of Fourier series, $B_n$'s are not orthonormal, but $\phi_n$'s are due to the normalization (see Lemma 3 in  \cite{Coif2}).

It was shown by Coifman and Peyri\'{e}re  \cite{Coif2}  that $(\phi _n)_{n=0}^{\infty}$ is a basis in the Hardy spaces $H^p( \ZT )$, $1 < p < \infty$ in the following sense. With the inner product defined for two functions $f, g$ on the unit circle $\ZT$ as
\begin{equation*}
\langle f, g \rangle = \frac{1}{2 \pi} \int_{-\pi }^{\pi} f(e^{i\theta }) \overline{g(e^{i \theta} )} d\theta,
\end{equation*}
we have that the MT series converges in $H^p$, i.e.
\begin{equation}\label{MTseries}
\sum_{n=0}^N \langle f , \phi_n \rangle \phi_n \xrightarrow[N\to \infty]{ H^p} f.
\end{equation}

There is a natural connection between the MT series and the non-linear phase unwinding decomposition introduced in the dissertation of Nahon  \cite{Nah}. For a function $F$ in $H^p(\mathbb{T})$ one can consider the Blaschke factorization $$F(z) - F(0) = B(z)F_1(z),$$ where $B$ is a Blaschke product and $F_1$ is in $H^p(\mathbb{T})$ and does not have zeros in $\mathbb{D}$. Iterating the procedure, one obtains the formal unwinding series
\begin{equation}\label{unwinding}
    F = F(0) B + F_1(0) B B_1 + \cdots + F_n(0) B B_1 \dots B_n + \cdots.
\end{equation}
Numerical simulations from   \cite{Nah} suggest that the right-hand side of \e{unwinding} converges back to the function and generally this happens at exponential rate. The result of Coifman and Peyri\'ere implies convergence in $H^p$. The case $p=2$ was previously obtained by Qian   \cite{Qian}, who had developed a similar theory to phase unwinding independently of Nahon in   \cite{Qian2}. Coifman and Steinerberger   \cite{Coif1} proved convergence in several different contexts including convergence in fractional Sobolev spaces $H^s$, $s>-\frac{1}{2}$ if the initial function $F$ is in $H^{s+\frac{1}{2}}$.

If at each step of the unwinding decomposition \e{unwinding} the Blaschke product $B_n$ has finitely many zeros, for example if $F$ is holomorphic in an $\epsilon$-neighborhood of the unit disk, we can consider the sequence of all zeros of $B, B_1, \dots$. The associated MT series will then reproduce the unwinding decomposition. Intuitively, making the zeros adapted to the function should accelerate the convergence. For this reason, the MT series is also called the Adaptive Fourier Transform.
For an overview of these constructions, we refer to the recent paper   \cite{Coif3}. For some further results and intuition on the unwinding decomposition we refer to   \cite{Steinerberger1, steinerberger2}.

We are interested in almost everywhere convergence of the MT series \e{MTseries}. By standard techniques, almost everywhere convergence can be deduced from estimates of the maximal partial sum operator. Denote
\begin{equation}\label{maxoperator}
Tf(e^{ix}) := T^{(a_n)} f ( e^{ix} ) := \sup_n | \sum_{n=0}^N \langle f , \phi_n \rangle \phi_n (e^{ix}) |.
\end{equation}

\begin{Question}
Is the maximal partial sum operator \e{maxoperator} bounded on $L^p$?
\end{Question}

If $a_n \equiv 0$, then the MT series reduces to the classical Fourier series and the operator \e{maxoperator} reduces to the Carleson operator. In this case the positive answer to the above question is given by the Carleson-Hunt theorem   \cite{Carleson, Hunt}.

We give two partial answers to this question in this paper. First, if the points are in a compact disc inside the open unit disc, the problem becomes a more benign perturbation of the Carleson-Hunt theorem. In this case, we quantify the $L^p$ norm of the maximal partial sum operator depending on the distance of the compact disc from the unit circle.
\begin{theorem}\label{compactcase}
Let $0< r<1$ and let $(a_n)_{n=1}^{\infty}$ be an arbitrary sequence such that $|a_n| \leq r$ for all $n$. Then, for $1< p < 2$,
\begin{equation}\label{logblowupp<2}
\| T^{(a_n)} \|_{L^p (\ZT ) \to L^p (\ZT ) } \lesssim_p \log \frac{1}{1-r}.
\end{equation}
For $2\leq p < \infty$, we have the better estimate
\begin{equation}\label{logblowupp>2}
\| T^{(a_n)} \|_{L^p (\ZT ) \to L^p (\ZT ) } \lesssim_p \sqrt{ \log \frac{1}{1-r} }.
\end{equation}
Furthermore, for $1< p \leq 2$, we have a lower bound in the sense, that for every $0 < r < 1$ there exists a sequence $(a_n)$ with $|a_n| \leq r$ such that
\begin{equation}\label{logblowupcounterexample}
    \| T^{(a_n)} \|_{L^p (\ZT ) \to L^p (\ZT ) } \gtrsim \sqrt{ \log \frac{1}{1-r} }.
\end{equation}
In particular, the bounds \e{logblowupp>2} and \e{logblowupcounterexample} are sharp for $p=2$.
\end{theorem}

There is a conformally invariant version of \trm{compactcase} for $p=2$ and arbitrary compact sets inside the disc. In that case, the quantity $\log \frac{1}{1-r}$ is replaced by the hyperbolic diameter of the compact set. In the Appendix, we prove this invariance and discuss the situation for $p$ other than $2$.

We turn to the second partial answer, the case when the points are in a non-tangential approach region to the boundary.
\begin{theorem}\label{nontangentialcase}
Let $a_n$ be inside the triangle with vertices $(1,0), (\frac{1}{2}, \frac{1}{2})$ and $(\frac{1}{2}, -\frac{1}{2})$ for all $n$, then
\begin{equation}
\| T \|_{L^2 (\ZT ) \to L^2 (\ZT )} \lesssim 1.
\end{equation}
\end{theorem}

\trm{compactcase} and \trm{nontangentialcase} turn out to be closely related to the polynomial Carleson theorem. Let us recall a special case of the polynomial Carleson operator   \cite{Lie1,Lie3, Pav}
\begin{equation*}
C_d f (x) := \sup_{Q} \sup_{0 < \epsilon < N} \big| \int_{\epsilon < |x-y| < N} f(y) e^{iQ(y)} \frac{dy}{x-y} \big|,
\end{equation*}
where the first supremum is taken over polynomials of degree at most $d$. The case $d= 1$ is the classical Carleson operator. Its weak $L^2$ bounds were implicit in Carleson's paper  \cite{Carleson} on almost everywhere convergence of the Fourier series, Hunt improved this to $L^p$ bounds, $1<p<\infty$. Alternative approaches appeared in Fefferman  \cite{Fef}, Lacey and Thiele  \cite{LT}. On the other hand, Stein and Wainger  \cite{SW} proved the case $d \geq 2$ but restricted to polynomials without the linear term. Lie combined the two techniques in  \cite{Lie1, Lie3} to prove the general $L^p$ bounds for $C_d$. Finally, Zorin-Kranich generalized the argument to higher dimensions and non-convolution Calderon-Zygmund kernels in  \cite{Pav}. For a gentle introduction to Carleson's theorem and a discussion of the different approaches we refer to Demeter's paper  \cite{Demeter}.

In \trm{compactcase}, we first obtain the logarithmic dependence \e{logblowupp>2} for all exponents by using the techniques of the polynomial Carleson theorem for the small scales and a triangle inequality for the large scales. Then, for $p=2$ we are able to improve the estimate for the large scales using a $TT^*$ argument and the analyticity of Blaschke products. The bound \e{logblowupp>2} for $p>2$ follows by black-boxing sparse domination results for Carleson-type opeators such as Theorems 9.1 and 9.2 in  \cite{Kar}. We refer to the bibliography of the latter for more references on sparse domination. Whether \e{logblowupp<2} and \e{logblowupp>2} are sharp we do not know.

\trm{nontangentialcase} is true for similar non-tangential approach regions to other points on the circle. Furthermore, if one takes the union of $k$ approach regions for $k$ distinct points at once, then it is possible to prove along the lines of \trm{nontangentialcase} that the $L^2$ norm of the operator \e{maxoperator} is bounded by $k$. If the boundary points for the approach regions are chosen to be equidistant, then the construction giving \e{logblowupcounterexample} also provides the lower bound $\sqrt{\log k}$. The sharp bound for this configuration is again unknown to us. It could also be interesting to consider approach regions to countably many points for various configurations.

The paper is organized as follows. In Section 2, we generalize the polynomial Carleson theorem in a way to suit our setting. In Section 3, we establish some properties for partial sums of MT series and for Blaschke phases. In Section 4, we prove the upper bounds of \trm{compactcase} and in Section 5 we prove \trm{nontangentialcase}. In section 6, we construct a counterexample for the lower bound of \trm{compactcase}. Finally, in the Appendix, we discuss the invariance of the problem under conformal maps.

We write $a \lesssim b$ if $a \leq c b$ for an absolute constant $c$, and $a\sim b$ if $a\lesssim b$ and $b\lesssim a$. If $c$ depends on parameters $\mathcal{C}$, we write $a \lesssim_{\mathcal{C}} b$.

\section*{Acknowledgments}
I am grateful to my advisor Christoph Thiele for posing the problem and constant support at every stage of this project. Also I thank Pavel Zorin-Kranich for numerous helpful discussion.

The author acknowledges support of the Hausdorff Center for Mathematics, funded by the DFG under Germany's Excellence Strategy - EXC-2047/1 - 390685813 as well as CRC 1060.

\section{Generalization of polynomial Carleson theorem}
In this section we first lay out an axiomatic approach to the polynomial Carleson theorem following  \cite{Pav} and restate it as \trm{Pavel}. We, then, give the list of modifications one has to make in  \cite{Pav} to prove our theorem. For that reason we also try to preserve the notation of  \cite{Pav}.

\subsection{General Setting}\label{generalsetting}
Assume $C_1 > 0$ is some constant. Let $K$ be a translation-invariant Calder\'on-Zygmund kernel on $\ZR$, that is, a function
$K : \ZR\setminus \{0\} \to \mathbb{C}$ such that for $x \neq 0$
\begin{align*}
& | K(x) | \leq C_1 |x|^{-1},\quad | K'(x) | \leq C_1 | x |^{-2},
\end{align*}
and the associated operator is bounded on $\mathcal{L} ( L^2 (\ZR) )$ by $C_1$. We further assume  that there exists an $C_2>0$ such that
\begin{equation}
   \supp K \subset [- C_2/2 , C_2/2 ].
\end{equation}

Let $\ZQ$ be a countable subset of $C^2( \ZR )$. For each interval $I$ and $P, Q \in \ZQ$ we define
\begin{equation}\label{metric}
d_I (P,Q) := \sup_{x,y \in I} | (P - Q)(x) - (P - Q)(y) |,
\end{equation}
and assume that $d_I$ is a metric on $\ZQ$.
We impose the following conditions on $\ZQ$. Assume that there exists a constant $C_0 > 0$ such that
\begin{enumerate}
\item[A.] (Lemma 2.6 in  \cite{Pav}) for any intervals $J \subset I$ with $|I| \leq C_2$ and $P, Q \in \ZQ$ we have
\begin{align}
& d_I (P, Q) \leq C_0 \frac{|I|}{|J|} d_J (P, Q), \\
& d_J (P, Q) \leq C_0 \frac{|J|}{|I|} d_I (P, Q); \label{nestedintervalsSmaller<Larger}
\end{align}
\item[B.] (John-Ellipsoid Property) for any $\lambda>1$ and interval $I$ with $|I| \leq C_2$, any $(\lambda, d_I )$-ball can be covered by $C_0 \lambda$ number of $(1, d_I )$-balls;
\item[C.]\label{Corput} (Lemma A.1 in  \cite{Pav}) for any measurable function $g : \ZR \to \mathbb{C}$, interval $J\subset \ZR$ with $\supp g \subset J $ and $|J| \leq C_2$, and any $Q, P \in \ZQ$ we have
\begin{equation*}
\big| \int_{J} e^{i (P - Q)(x)} g (x) dx \big| \leq C_0 \sup_{| y | < (1 + d_J (P, Q) ) ^{-1} |J| } \int _{\ZR} | g (x) - g (x-y)| dx.
\end{equation*}
\end{enumerate}

\begin{theorem}\label{Pavel}
Assume conditions $A,B$ and $C$ hold for the set $\ZQ$. We define the operator $T:\: L^1_{loc}(\ZR ) \to L^0 (\ZR)$ as
\begin{equation}\label{operator}
Tf (x) := \sup_{Q \in \ZQ } \sup_{0 < \epsilon < N } | \int_{\epsilon < |y| <N} f(y) e^{iQ(y)} K(x-y) dy|.
\end{equation}
Further, let $0 \leq \alpha < 1/2$ and $0\leq \nu, \kappa < \infty$. Let $F,G \subset \ZR$ be measurable subsets and $\tilde{F} := \{ M \Z1_F > \kappa \}, \tilde{G} := \{ M \Z1 _G > \nu \}$. Then the following inequalities hold(with the implicit constants independent of $M$)
\begin{align}
\| T \|_{L^2 (\ZR) \to L^2(\ZR)} & \lesssim_{C_1, C_0} 1, \label{global}\\
\| \Z1_G T \Z1_{\ZR\setminus \tilde{G}} \|_{L^2 (\ZR) \to L^2(\ZR)} & \lesssim_{C_1, C_0, \alpha} \nu^{\alpha}, \label{localized1}\\
\| \Z1_{\ZR \setminus \tilde{F} } T \Z1_F \|_{L^2 (\ZR) \to L^2(\ZR)} & \lesssim_{C_1, C_0, \alpha} \kappa^{\alpha}. \label{localized2}
\end{align}
\end{theorem}

\trm{Pavel} is a version on Zorin-Kranich's Theorem 1.5 in  \cite{Pav}. Lemma B.1 in the latter establishes $L^p$ boundedness of $T$, for $1 < p<\infty$, from the inequalities \e{localized1} and \e{localized2}. Furthermore, Lemma B.2 establishes the local inequalities \e{localized1} and \e{localized2} for a Hardy-Littlewood maximal operator. Thus, we will allow a loss of maximal operator to be able to alternate between different kernels of the Hilbert transform.

\subsection{List of modifications in the proof}
We start the comparison with  \cite{Pav}. We want to mention here, that we refer to the equation numbering of the 5th version of Zorin-Kranich's paper on arxiv. On the left column we indicate the sections in  \cite{Pav} that we make modifications in, and in {\it italic} symbols we refer to the numbering of  \cite{Pav}.
\begin{enumerate}
\item[Sec.1.] We put $\mathbf{d} =1$ as we are interested in dimension $1$,
\item[Sec.1.] and let the kernel $K$ be translation-invariant and $1$-H\"older continuous, i.e. $\tau = 1$.
\item[Sec.2.] In  \cite{Pav}, $\ZQ$ is the vector space of polynomials up to degree $d$, for us it is a countable subset of $C^2(\ZR)$ with conditions $A,B$ and $C$.
\item[Sec.2.] In  \cite{Pav}, we have $C_2=+\infty$. Note, that if $C_2 < + \infty$, then the operator decomposes into scales smaller than $\log C_2$ and all the intervals considered in the proof will have lengths less than $C_2$.
\item[Sec.2.] The existence of a countable dense subset in a finite-dimensional vector space of polynomials w.r.t. to the supremum norm allows us to restrict the supremum in the definition of $T$ to finite number of polynomials on {\it the first paragraph of page 4}. In our case, $\ZQ$ itself is countable.
\item[Sec.2.1.] In  \cite{Pav}, the $\| \cdot \|_{I}$-norm is used only as a metric, i.e. for the difference of two elements of $\ZQ$ and never for a single element or a scalar multiple of an element. Hence, we replace the definition of the norm in {\it equation 2.5} by \e{metric}, which is the same, and assume that the latter is a metric.
\item[Sec.2.1.] Condition A replaces {\it Lemma 2.6 }. The power $d$ on the right-hand side of {\it the first inequality}, that is $1$ in our case, has no effect, as the inequality is only used to obtain {\it inequality 2.10} which is valid also under our assumption.
\item[Sec.3.2.] In {\it Definition 3.17 }, we replace $\dim \ZQ$ by $1$. This turns out to be compatible with Condition B. The idea is that instead of a $d$-dimensional vector space of polynomials we have a $1$-dimensional manifold, a curve, of functions with a John-Ellipsoid covering property.
\item[Sec.3.2.] Due to the last point, we make the same substitution in {\it inequality 3.21}. The covering argument, that comes after, is due to Condition B.
\item[Sec.4.1.] In the statement of {\it Lemma 4.1} and in all following occasions, we substitute $d = 1$. We mentioned in the beginning that $\tau =1$.
\item[Sec.4.1.] In the proof of {\it Lemma 4.1 } and in all the subsequent occasions, instead of {\it Lemma A.1 } we use Condition C.
\item[Sec.5.3.] Item \e{nestedintervalsSmaller<Larger} of Condition A is used in {\it the inequality after (5.19)}.
\item[Sec.5.3.] Then again, it appears in {\it the proof of claim (5.22)} and in {\it the block of equations after (5.24)}.
\end{enumerate}

\section{Partial Sums and M\"obius Phases}
In the first subsection of this section, we write the MT partial sum operator in a closed form and see that operator \e{maxoperator} is, up to a Hilbert transform, a maximally modulated Hilbert transform. In the second subsection, we establish two properties for the phases of M\"obius transforms that connects our problem to the general setting of the previous section.

\subsection{Partial Sum Operator}
First, we want to pass to the notation that is more conventional in stationary phase literature and in  \cite{Pav}. We write the M\"obius transform as
\begin{align*}
    \frac{e^{ix} - b}{1-\overline{b} e^{ix}} & = e^{ix}\cdot \frac{1-be^{-ix}}{\overline{ 1- be^{-ix} }}
    = e^{ix } \frac{ e^{ i \arg (1 - b e^{-ix}) } }{ e^{ - i \arg (1 - b e^{-ix}) } } \\
    & = e^{i( x + 2 \arg ( 1 - b e^{-ix} ) )}.
\end{align*}
Denoting $b = re^{i\beta}$, we apply twice the Euler formula for the expression in the argument above.
\begin{align*}
    1 - b e^{-ix} & = 1 - r e^{i(\beta - x)} \\
    & = |1 - r e^{i (\beta -x )} | \left( \frac{1 - r \cos (\beta - x)}{ |1 - r e^{i (\beta -x )}| } - i \frac{ r\sin (\beta - x)}{ |1 - r e^{i (\beta -x )}| } \right) \\
    & = |1 - r e^{i (\beta -x )} | e^{ i \arcsin \frac{ r \sin ( x - \beta ) }{ \sqrt{ 1 + r^2 - 2 r \cos( x - \beta) } }}.
\end{align*}
Here we have used that $1-r\cos(\beta - x)$ is always positive so that we can rightfully write the arcsine. Thus, putting
\begin{equation}\label{phasedefined}
\Psi_b (x) := x - \arg (b) + 2 \arcsin \frac{ | b | \sin ( x - \arg (b) ) }{ \sqrt{ 1 + |b|^2 - 2 |b| \cos( x - \arg (b) ) } },
\end{equation}
we conclude
\begin{equation}\label{MobiusPhase}
    \frac{e^{ix} - b}{1-\overline{b} e^{ix}} = e^{i \Psi_b (x) + i \arg (b)}.
\end{equation}

Let us denote by $\theta_N$ the phase of the Blaschke product $B_N$, that is $\theta _N := \sum\limits_{j=1}^N \Psi_{a_j}$. Next, we claim that the partial sum of Malmquist-Takenaka series can be written in a closed form, namely,
\begin{equation}\label{operatorclosedform}
\sum_{n=0}^{N-1} \langle f , \phi_n \rangle \phi_n (e^{ix})= \frac{1}{2\pi} \int\limits_{-\pi}^{\pi} f( e^{iy} )
 \frac{ e^{ i(\theta_N(x) - \theta_N (y) ) } - 1 }{ e^{ i(x - y) } - 1 } dy.
\end{equation}

Formula \e{operatorclosedform} appears in  \cite{Coif2} in a more general setting of invariant subspaces of $H^2$. The orthogonal complement of the invariant subspace $B_N H^2$ is spanned by $(\phi_n)_{n=0}^{N-1}$. Thus, the partial sum operator of the MT series equals $I - P$, where $I$ is the identity operator and $P$ is the orthogonal projection of $L^2$ onto $B_N H^2$. Let $\mathcal{H} : L^2 \to L^2$ be the projection operator on subspace $H^2$. Then, on $H^2$, the identity operator $I$ coincides with $\mathcal{H}$. On the other hand, $P$ has the form $\chi_{B_N} \mathcal{H} \chi_{B_N^{-1}}$, where $\chi_g$ is the pointwise multiplication operator by the function $g$.

However, for the convenience of the reader, we give a self-contained proof of this fact by induction.
\begin{proof}[Proof of \e{operatorclosedform}]
First, we check the following chain of identities, where in the first line we apply \e{MobiusPhase}.
\begin{align*}
    \frac{ e^{i ( \Psi_{b} (x) - \Psi_{b} (y) ) } - 1}{ e^{i(x-y)} - 1} & = \frac{ (e^{ix} - b)(1 - \overline{b} e^{iy} ) - ( 1 - \overline{b} e^{ix} )( e^{iy} - b ) }{ (e^{i(x - y)} - 1) (1 - \overline{b} e^{ix} ) ( e^{iy } - b )  } \\ 
    & = \frac{ e^{ix} + |b|^2 e^{iy} - e^{iy} - |b|^2 e^{iy} }{ ( e^{i(x - y)} -1)(1 - \overline{b} e^{ix} ) ( e^{iy } - b )  } \\
    & = \frac{1-|b|^2}{( 1- \overline{b} e^{ix} ) \overline{(1 - \overline{b} e^{iy}) } }. \numberthis\label{calculations}
\end{align*}
Equality \e{operatorclosedform} for $N=1$ follows by the above identity with $b = a_1$ and by the definition \e{BlaschkeMT} of the function $\phi_{n}$. Namely,
\begin{align*}
    \langle f , \phi_0 \rangle \phi_0 & = \frac{ \sqrt{ 1 - | a_1 |^2 } }{ 1 - \overline{a_1} e^{ix} } \frac{1}{2\pi} \int\limits_{-\pi}^{\pi} f( e^{iy} )
 \overline{ \left( \frac{ \sqrt{ 1 - | a_1 |^2 } }{ 1 - \overline{a_1} e^{iy} } \right) } dy \\
 & = \int\limits_{-\pi}^{\pi} f( e^{iy} ) \frac{ e^{i ( \Psi_{a_1} (x) - \Psi_{a_1} (y) ) } - 1}{ e^{i(x-y)} - 1} dy.
\end{align*}
Assuming the formula \e{operatorclosedform} for $N$, we write
\begin{align*}
    \frac{1}{2\pi} \int\limits_{-\pi}^{\pi} f( e^{iy} )
 \frac{ e^{ i(\theta_{N+1} (x) - \theta_{N+1} (y) ) } - 1 }{ e^{ i(x - y) } - 1 } dy
  = \frac{1}{2\pi} \int\limits_{-\pi}^{\pi} f( e^{iy} )
 \frac{ e^{ i(\theta_{N} (x) - \theta_{N} (y) } - 1 }{ e^{ i(x - y) } - 1 } dy \\
 + \frac{1}{2\pi} \int\limits_{-\pi}^{\pi} f( e^{iy} ) 
 \frac{ e^{ i(\theta_{N+1} (x) - \theta_{N+1} (y) ) } - e^{ i(\theta_{N} (x) - \theta_{N} (y) ) } }{ e^{ i(x - y) } - 1 } dy \\
 = \sum_{n=0}^{N-1} \langle f , \phi_n \rangle \phi_n (e^{ix})
 +\frac{1}{2\pi} e^{ i \theta_N (x)} \int\limits_{-\pi}^{\pi} f( e^{iy} ) e^{- i\theta_N(y)} 
 \frac{ e^{ i( \Psi_{a_{N+1}} (x) - \Psi_{a_{N+1}} (y) ) } - 1 }{ e^{ i(x - y) } - 1 } dy \\
 = \sum_{n=0}^{N} \langle f , \phi_n \rangle \phi_n (e^{ix}).
\end{align*}
In the penultimate equality we have used the induction hypothesis, and in the last line, we have used \e{calculations} for $b =a _{N+1}$ and the definition of $\phi_n$'s.
\end{proof}

Using \e{operatorclosedform}, the operator \e{maxoperator} can be rewritten up to a Hilbert transform on the circle as
\begin{equation}\label{maxoperator2}
Tf(e^{ix}) = \sup_n \big| \int\limits_{-\pi}^{\pi} f( e^{iy} )
e^{ - i \theta_n (y) } \frac{ dy }{ e^{ i(x - y) } - 1 } \big|.
\end{equation}

\subsection{Two Lemmas about M\"obius phases}
We identify the interval $[-\pi, \pi]$, $\ZR / 2\pi \ZZ$ and the unit circle $\ZT$ through the natural parametrizations and consider a lacunary decomposition of the circle adapted to a point $b$ inside the unit disc
\begin{equation}\label{partition}
\ZT = \bigcup\limits_{j=0}^{N-1} J^b_j,
\end{equation}
where $J^b_j = [\arg (b) + t_j, \arg (b) + t_{j+1} ] \cup [\arg (b) - t_j, \arg (b) - t_{j+1} ]$ and $t_0 = 0$, $t_j = 2^{j} (1-|b|)$ for $1\leq j < N$, $t_N = \pi$ and $N = [ \log_2 \frac{1}{1-|b|} ]$ ($[x]$ denotes the integer part of $x$).

\begin{lemma}\label{decomp}
There exist absolute constants $A, B>0$ such that for any point $b$ inside the unit disk and $0\leq j < N$ it holds
\begin{equation}\label{derratior}
\frac{A}{1-|b|} 2^{-2j} \leq \inf_{y\in J^b_j } \Psi_{b} ' (y) \leq \sup_{y \in J^b_j} \Psi_{b} ' (y)  \leq \frac{B}{1-|b|} 2^{-2j}.
\end{equation}
\end{lemma}

\begin{proof}
First of all, we compute
\begin{equation}\label{phasederiv}
\Psi_b ' (y) = \frac{ 1 - |b|^2 }{ 1 + |b|^2 - 2|b| \cos ( y - \arg(b) ) },
\end{equation}
and in particular $\Psi_b' \gtrsim 1 - |b| >0 $.

Also as $\Psi_b (x) = \Psi_{|b|} (x - \arg b)$, we see that the inequality \e{derratior} is rotationally invariant. We can assume that $b = |b| =: r$.

Next, $t_{N-1} \geq ( 1-r) \cdot \frac{1}{4(1-r)} =1/4 $. Thus, for any $y \in I_{N-1}$ we have
\begin{equation*}
    \frac{1-r}{1+r} \leq \Psi_b ' (y) \leq \frac{1-r^2}{1+r^2-2r \cos{1/4}} \lesssim 1-r.
\end{equation*}
Therefore, the conclusion of the lemma is true for $j= N-1$. One can also check the conclusion of the lemma for $j=0$ and we restrict our attention to intervals $I_j := [t_j, t_{j+1}]$ with $1\leq j < N-1$.

Let $c_1 > c_2>0 $ be constants such that $ 1 - c_1 x^2 < \cos (x) < 1 - c_2 x^2$ for $x \in [- \pi/2, \pi/2]$. Then,
\begin{align*}
\sup_{y \in I_j} \Psi_{r} ' (y) &= \sup_{y\in I_j} \frac{1 - r^2 }{1 + r^2 - 2 r \cos ( y ) } \\
& \leq \frac{ 1 - r }{1 + r^2 - 2r ( 1 - c_2 2^{2j} (1-r)^2 ) } \\ 
& = \frac{1}{(1-r) (1 + 2r c_2 2^{2j} ) } \lesssim \frac{2^{-2j}}{1-r}.
\end{align*}
In last inequality we have assumed $r > 1/2$ because otherwise there is nothing to prove. The other inequality is analogous
\begin{align*}
\inf_{y \in I_j} \Psi_{r} ' (y) &= \inf_{y\in I_j} \frac{1 - r^2 }{1 + r^2 - 2 r \cos ( y ) } \\
& \geq \frac{ 1 - r }{1 + r^2 - 2r ( 1 - c_1 2^{2(j+1)} (1-r)^2 ) } \\ 
& = \frac{1}{(1-r) (1 + 2r c_1 2^{2(j+1)} ) } \gtrsim
\frac{2^{-2j}}{1-r}.
\end{align*}
\end{proof}

The next lemma is a Van der Corput-type estimate.
\begin{lemma}\label{oscillatory}
Let $r> 1/2$ and $I$ be an interval of length at most $(1-r)$. Let $b_1, \dots, b_k$ be points in the unit disk with the property that
\begin{equation}\label{distancecondition}
if\; \arg (b_j) \in I , \text{ then } 1- | b_j | \geq c(1 - r),
\end{equation}
for some parameter $c>0$. Then for any $g \in C^1(I)$
\begin{equation}\label{oscestimate}
| \int_I e^{ i \sum\limits_{j=1}^k \Psi_{b_j} (y)} g(y) dy | \lesssim_c \frac{1}{ \sum\limits_{j=1}^k \inf_{y \in I } \Psi_{b_j} '(y)  } ( \| g' \|_{L^1(I)} + \frac{1}{|I|} \int_I | g | ).
\end{equation}
\end{lemma}

\begin{proof}
To begin, assume that $g\equiv 1$ and compute
\begin{equation}\label{secondder}
\Psi '' _b (y) = -2|b| \frac{ ( 1 - |b|^2 ) \sin ( y - \arg(b) ) }{ ( 1 + |b|^2 - 2|b| \cos( y - \arg(b) ) )^2 }.
\end{equation}
Let $\phi(y) := \sum\limits_{j=1}^k \Psi_{b_j} (y)$. We apply the usual integration by parts.
\begin{align*}
| \int_I e^{i \phi(y)} dy | & \leq | \int_{I} e^{i \phi(y) } \frac{ \phi '' (y) }{ ( \phi ' (y) )^2 } dy| + | ( \phi ' )^{-1} e^{i\phi} \Big\vert_I |.
\end{align*}
The second summand is bounded by the right-hand side of \e{oscestimate}. By translation symmetry (rotation of the circle) let $I = [0, x_0) $ with $x_0 < 1 - r$. We split the points into four groups.
\begin{align*}
A_1 & := \{ j \: : \: \arg(b_j) \in [x_0, \pi )  \}, \\
A_2 & := \{ j \: : \: \arg(b_j) \in [-\pi+x_0, 0) \}, \\
A_3 & := \{ j \: : \: \arg(b_j) \in [0, x_0 )  \}, \\
A_4 & := \{ j \: : \: \arg(b_j) \in [-\pi, -\pi + x_0 )  \}.
\end{align*}

Estimates for $A_1$ are $A_2$ identical. So we consider only $A_1$. By \e{secondder} $ \Psi_{b_j} ''(y)$ maintains the sign on $I$ for $j \in A_1$. Hence,
\begin{align*}
| \int_I e^{i \phi(y)} \frac{ \sum\limits_{j \in A_1 } \Psi _{b_j} ''(y) }{ ( \sum\limits_{j=1}^k \Psi_{b_j}'(y) )^2 } dy| &
 \leq \int_I \frac{ \sum\limits_{j \in A_1 } \Psi _{b_j} ''(y) }{ ( \sum\limits_{j \notin A_1} \inf_{y\in I} \Psi_{b_j}' + \sum\limits_{j\in A_1} \Psi_{b_j}'(y) )^2 } dy \\
 & = \frac{1}{\sum\limits_{j \notin A_1} \inf_{y\in I} \Psi_{b_j}' + \sum\limits_{j\in A_1} \Psi_{b_j}'(\cdot)} \Big\vert_I \\
 & \leq \frac{1}{\sum\limits_{j =1}^k \inf_{y\in I} \Psi_{b_j}' }.
\end{align*}

We turn to $A_4$. Let $b$ be a point inside the disc such that $\arg (b) \in [-\pi, -\pi + x_0)$ and $y \in I$. Then, recalling that $|I|=x_0 \leq 1-r$ and the formula \e{phasederiv}, we have
\begin{align*}
    | \Psi_{b} ''(y) | &\leq \frac{1 - |b|^2}{( 1 + |b|^2 - 2|b| \cos( y - \arg(b) ) )^2} \\
    & \leq \Psi_b '(y) \frac{1}{1 + |b|^2 -2|b| \cos (\pi - 2(1-r) )} \leq \Psi_b'(y).
\end{align*}
Thus, we conclude
\begin{align*}
   | \int_I e^{i \phi(y)} \frac{ \sum\limits_{j \in A_4 } \Psi _{b_j} ''(y) }{ ( \sum\limits_{j=1}^k \Psi_{b_j}'(y) )^2 } dy| &
 \leq \int_I \frac{ \sum\limits_{j \in A_4 } |\Psi _{b_j} ''(y) | }{ ( \sum\limits_{j \notin A_4} \inf_{y\in I} \Psi_{b_j}' + \sum\limits_{j\in A_4} \Psi_{b_j}'(y) )^2 } dy \\
 & \leq \frac{1-r}{\sum_j \inf_{y\in I} \Psi_{b_j}'}.
\end{align*}

Finally, we treat the sum with $A_3$. We need to further decompose $A_3$. For $m\geq \log (c)$ let
\begin{equation*}
B_m := \{ j \in A_3 \: : \: 2^m (1-r) \leq 1 - |b_j| < 2^{m+1} (1-r) \}.
\end{equation*}
Then, for $y \in I$ and $j\in B_m$ we have
\begin{align*}
    1 + |b_j|^2 - 2 |b_j| \cos (y - \arg (b_j) ) & \leq 1 + |b_j|^2 - 2 |b_j| \cos (1-r) \\
    & \leq 1 + |b_j|^2 - 2 |b_j| ( 1 - c_1 (1-r)^2 ) \\
    & \sim (1 - b_j )^2.
\end{align*}
On the other hand,
\begin{align*}
    1 + |b_j|^2 - 2 |b_j| \cos (y - \arg (b_j) ) \geq 1 + |b_j|^2 - 2 |b_j| = (1-b_j)^2.
\end{align*}
So we conclude, that
\begin{equation*}
    1 + |b_j|^2 - 2 |b_j| \cos (y - \arg (b_j) ) \sim (1-b_j)^2.
\end{equation*}
Using this equivalence, we bound the the first and second derivatives of $\Psi_{b_j}$'s.
\begin{equation*}
\Psi '_{b_j} (y) \gtrsim \frac{ 1 - |b_j| }{ (1 - b_j)^2 } = \frac{ 1 }{(1-r)2^{m}}.
\end{equation*}
For the second derivative we write
\begin{equation*}
|\Psi '' _{b_j} (y)| \lesssim \frac{ (1-r) (1-|b_j|) }{ (1 - |b_j| )^4 } \lesssim \frac{1}{(1-r)^2 2^{3m}}.
\end{equation*}
We are ready to estimate the integral for $B_m$. Putting $C = \sum\limits_{j \notin B_m} \inf_{y\in I} \Psi_{b_j}'$ we write
\begin{align*}
| \int_I e^{i \phi(y)} \frac{ \sum\limits_{j \in B_m } \Psi _{b_j} ''(y) }{ ( \sum\limits_{j=1}^k \Psi_{b_j}'(y) )^2 } dy| &
\leq \int_I \frac{ \sum\limits_{j \in B_m } |\Psi _{b_j} ''(y)| }{ (C + \sum\limits_{j\in B_m} \Psi_{b_j}'(y) )^2 } dy \\
& \lesssim \int_I \frac{ \#(B_m)  \frac{1}{(1-r)^2 2^{3m}} }{ ( C + \# ( B_m ) \frac{1}{(1-r) 2^{m}} ) (C + \sum\limits_{j\in B_m} \inf \Psi '_{b_j} ) } \\
& \lesssim \frac{1}{2^{2m}} \cdot \frac{1}{C + \sum\limits_{j\in B_m} \inf \Psi '_{b_j}} = \frac{1}{2^{2m}} \cdot \frac{1}{ \sum\limits_{j = 1}^k \inf \Psi '_{b_j} },
\end{align*}
where in the last inequality we have used $C>0$ and $|I| \leq 1-r$.
Summing over $m \geq \log (c)$ finishes the case $g \equiv 1$.

The general case follows by a standard integration by parts argument. Denote $F(x) := \int_0^x e^{i \sum_{j=1}^k \Psi_{b_j} (y) } dy$ for $0\leq x\leq x_0$. We have proved
$|F(x)| \lesssim \frac{1}{\sum_{j=1}^k \inf_I \Psi_{b_j}'}$. Thus, we conclude
\begin{align*}
    | \int_0^{x_0} g(y) e^{i\sum_{j=1}^k \Psi_{b_j} (y)} dy | &= | \int_0^{x_0} g(y) F'(y) dy | \\
    &\leq | g(0)F(0) - g(x_0)F(x_0) | + | \int_0^{x_0} F(y) g'(y) dy| \\
    & \lesssim \frac{1}{\sum_{j=1}^k \inf_I \Psi_{b_j}'} (\int_0^{x_0} |g'(y)|dy + g(x_0) ).
\end{align*}
\end{proof}

\section{Upper bound for \trm{compactcase}}
Assume $0 < r < 1$ is close enough to $1$ and $| a_n | \leq r$. Let $N : \ZT \to \mathbb{N}$ be a choice function and
\begin{equation}\label{maxoperator10}
T f( x ) := \int\limits_{-\pi}^{\pi} f( y )
e^{ - i \theta_{N(x)} (y) } \frac{ dy }{ \sin \frac{x-y}{2} }
\end{equation}
be the linearzied version of the operator \e{maxoperator2} up to a Hardy-Littlewood maximal operator. Denoting $R:= \log \frac{1}{1-r}$, we need to prove for $1 < p \leq 2$
\begin{equation}\label{nonsharp}
    \| T f \|_{L^p( \ZT ) } \lesssim_p R \| f \|_{L^p( \ZT ) },
\end{equation}
and for $2 \leq p < \infty$
\begin{equation}\label{sharpp>2range}
    \| T f \|_{L^p( \ZT ) } \lesssim_p R^{\frac{1}{2} } \| f \|_{L^p( \ZT ) }.
\end{equation}

Let us decompose the kernel $\frac{1}{ \sin \frac{x-y}{2} }$ into scales. Take a bump function $\xi_0 \in C^{\infty} (\ZR)$ supported in $[-\frac{1}{8}, -\frac{1}{2} ] \cup [ \frac{1}{8} , \frac{1}{2} ]$, such that $\sum_{s = 0}^{\infty } \xi _0 (2^s \cdot ) \equiv 1$ on $[-\frac{1}{4}, \frac{1}{4}]$.
For any $x,y \in \ZR$, we have
\begin{align*}
\frac{1}{\sin \frac{x-y}{2}} & = \frac{1}{\sin \frac{x-y}{2}} \sum_{s = R - 5}^{\infty } \xi_{s} (x-y) + \frac{1}{\sin \frac{x-y}{2}} \Big(1 - \sum_{s = R - 5 }^{\infty } \xi_s (x-y) \Big) \\
& =: K_r(x-y) + \frac{1}{\sin \frac{x-y}{2}} \Big(1 - \sum_{s = R - 5 }^{\infty } \xi_s (x-y) \Big),
\end{align*}
where $\xi_s (\cdot ) = \xi _0 (2^s \cdot )$. $K_r$ is a Cald\'eron-Zygmund kernel on the real line supported on $[-c(1-r); c(1-r)]$ for some absolute constant $c$, and with the corresponding $C_1$ quantity from \sect{generalsetting} bounded by an absolute constant. We denote
\begin{align}
    & T_{-1} f ( e^{ix} ) := \int\limits_{-\pi}^{\pi} f( e^{iy} ) e^{ - i \sum_{j=1}^n \Psi_{a_j} (y) }  \Big(1 - \sum_{s = 0 }^{\infty } \xi_s (x-y) \Big) \frac{dy}{\sin \frac{x-y}{2}} , \\
    & T_{s} f ( e^{ix} ) := \int\limits_{-\pi}^{\pi} f( e^{iy} ) e^{ - i \sum_{j=1}^n \Psi_{a_j} (y) }  \xi_s (x-y) \frac{dy}{\sin \frac{x-y}{2}} , \\
    & T_{\text{Small}} f( e^{ix} ) := \sum\limits_{s = R-5}^{\infty} T_{s} f( e^{ix} ) = \int\limits_{-\pi}^{\pi} f( e^{iy} ) e^{ - i \sum_{j=1}^n \Psi_{a_j} (y) }K_r(x-y) dy, \\
    & T_{\text{Large}} f ( e^{ix} ) := \sum\limits_{s = -1}^{R+4} T_{s} f( e^{ix} ).
\end{align}
Then, we have $T = T_{\text{Small}} + T_{\text{Large}}$. For one scale, we have $| T_s f ( x ) | \lesssim Mf ( x )$, where $Mf$ is the Hardy-Littlewood maximal function. Thus, by a triangle inequality, we write
\begin{equation}\label{largescalestriangle}
    \| T_{\text{Large}} f \|_{L^p (\ZT ) } \lesssim_p R \| f\|_{L^p( \ZT )}.
\end{equation}
We prove in \subsect{smallscales} that
\begin{equation}\label{smallscalesestimate}
    \| T_{\text{Small}} f \|_{L^p( \ZT )} \lesssim_p \| f\|_{L^p( \ZT )}.
\end{equation}
Then, the inequality \e{nonsharp} for all $p$ follows immediately from \e{largescalestriangle} and \e{smallscalesestimate}. The improved estimate \e{sharpp>2range} will follow by a more subtle argument for the large scales. This is done in \subsect{improvedlocalestimate} and \subsect{allscalespgeq2}.

\subsection{Small Scales}\label{smallscales}
We make a transition into the real line to be able to apply \trm{Pavel}. First, we extend the phases outside $[-\pi, \pi]$.
\begin{equation*}
\psi_{a_n} (x) := \left\{ \begin{aligned}
& \Psi_{a_n}(x), \text{ if } x \in [-\pi, \pi) \\
& \Psi_{a_n}(\pi) + \Psi_{a_n} ' (\pi ) ( x-\pi), \text{ if } x \geq \pi \\
& \Psi_{a_n}(-\pi) + \Psi_{a_n} ' (-\pi) (x+\pi), \text{ if } x < -\pi .
\end{aligned}\right.
\end{equation*}
Then, recalling the remark regarding $K_r$ right after its definition we define
\begin{equation}
    \tilde{T}_{\text{Small}} f(x) := \int\limits_{\ZR} f( y ) e^{ - i \sum_{j=1}^n \psi_{a_j} (y) } K_r(x-y) \chi_{ [-\pi , \pi] } (x-y) dy.
\end{equation}
Let $f: \ZT \to \mathbb{C}$ is supported on a half circle and $\tilde{f} : \ZR \to \mathbb{C}$ such that $\tilde{f} (x) = f(e^{ix})$ for $x\in [-\pi, \pi]$ and $\tilde{f} (x) = 0$ elsewhere. Then, $\tilde{T}_{\text{Small}} \tilde{f} (x) =T_{\text{Small}} f ( e^{ix} )$ for $x\in [-\pi, \pi]$. \e{smallscalesestimate} will follow from the $L^p(\ZR )$ bound of $\tilde{T}_{\text{Small}}$ and a triangle inequality.

Define
\begin{equation*}
\ZQ := \{ \sum_{j=1}^n \psi_{a_j} \: : \: n \in \mathbb{N} \} \text{ and } C_2 := c(1-r)/2.    
\end{equation*}
Then, \trm{Pavel} applies to operator $\tilde{T} _{\text{Small}}$ as soon as we verify conditions A, B and C which we do next.

By \lem{decomp} and as $|a_n| \leq r$, we have, for any interval $I$ of length at most $c (1-r)$,
\begin{equation}\label{derequiv}
\sup_I \psi_{a_j} '\sim \inf_I \psi_{a_j} '.
\end{equation}
This equivalence is central for the arguments below.

\begin{enumerate}
\item[A.] Let $P = \sum_{j=1}^n \psi_{a_j}$ and $Q = \sum_{j=1}^m \psi_{a_j}$ with $n<m$ and $J \subset I$ with $|I| \leq c(1-r)$.  Then,
\begin{align*}
d_I (P, Q) &\leq |I| \sum_{j=n+1}^m \sup_I \psi_{a_j}' \lesssim \frac{|I|}{|J|} |J| \sum_{j=n+1}^m \inf \psi_{a_j}' \\
& \leq \frac{|I|}{|J|} \sup_{x,y\in J} |(P-Q)(x) - (P-Q)(y)| = \frac{|I|}{|J|} d_J (P,Q).
\end{align*}
The reverse inequality follows by the same argument.
\item[B.] Let $P$ and $I$ be as above. We look at
\begin{equation*}
    B_I (P, \lambda) := \{ Q \in \ZQ \: : \: d_I(P, Q) < \lambda \}.
\end{equation*}
By \e{derequiv} and the definition \e{metric} of the metric $d_I$, there are absolute constants $D_1, D_2 >0$ such that
\begin{align*}
    B_I (P, \lambda) & \subset \{ Q = \sum_{j=1}^m \psi_{a_j} \: : \: \sum_{j=n+1}^m \sup_I \psi_{a_j}' \leq D_1\lambda / |I| \}, \\
     B_I (P, 1) & \supset \{ Q = \sum_{j=1}^m \psi_{a_j} \: : \: \sum_{j=n+1}^m \sup_I \psi_{a_j}' \leq D_2 / |I| \}.
\end{align*}

Let $N >n$ be the largest index for which $\sum_{j=n}^{N+1} \sup_I \psi_{a_j}' \leq D_1 \lambda /|I|$.
Similarly, define $\tilde{N} < n$ to be the smallest index satisfying the above inequality. Then,
\begin{equation*}
     B_I (P, \lambda) \subset \{ Q = \sum_{j=1}^m \psi_{a_j} \: : \: \tilde{N} \leq m \leq N \}.
\end{equation*}
Now choose indices $n = n_1<n_2< \cdots < n_C =N $ consecutively so that
\begin{equation}\label{choiceofballs}
    \sum_{j=n_k}^{n_{k+1}-1} \sup_I \psi_{a_j}' \leq D_2 / |I| < \sum_{j=n_k}^{n_{k+1}} \sup_I \psi_{a_j}'.
\end{equation}
The set inclusions above imply that each set $\{ \sum_{j=1}^m \psi_{a_j} \: : \: n_k \leq m < n_{k+1} \}$ is in a $(1, d_I )$-ball. Namely,
\begin{equation*}
    \{ \sum_{j=1}^m \psi_{a_j} \: : \: n_k \leq m < n_{k+1} \} \subset B_I \big( \sum_{j=1}^{n_k} \psi_{a_j} , 1 \big).
\end{equation*}
Furthermore, summing up the left-hand sides of \e{choiceofballs} for $k = 1, \dots , C$ we conclude
\begin{equation*}
    C D_2/ |I| < 2 \sup_{j=n}^N \psi '_{a_j} \leq 2 D_1 \lambda/|I|.
\end{equation*}
Similarly, the same argument holds "from the left of $n$" for $\tilde{N}$. Hence, the number $C$ of $(1, d_I)$-balls, that are required to cover the $(\lambda, d_I )$-ball, is at most $\frac{4D_1}{D_2} \lambda$.

\item[C.] This property follows from \lem{oscillatory}. We reproduce the argument of Lemma A.1 of  \cite{Pav}.

Let $g$, $J, P, Q$ be as in the hypothesis of Condition C, and $\Delta := d_J(P, Q) + 1$. If $\Delta \leq 3/2$, then as $\supp g \subset J$ we trivially have
\begin{align*}
| \int_J e^{i(P-Q)(x)} g(x) dx| & \leq \int_J |g(x)| dx \leq \sup\limits_{ |y| < 2|J|/3 } \int_{ \ZR } | g(x) - g(x-y) | dx.
\end{align*}
Otherwise, assume $\Delta \geq 3/2$ and denote $\beta := \sup\limits_{ |y| < \Delta ^{-1} |J| } \int_{ \ZR } | g(x) - g(x-y) | dx.$
Let $\chi_0$ be a smooth bump function supported in $(-1,1)$ with integral $1$, and $\chi := \Delta |J|^{-1} \chi_0 (\Delta |J|^{-1} \cdot)$ is its dilate. We want to change $g$ by $\tilde{g } := \chi * \psi$. For the error term we have
\begin{align*}
\int | g(x) - \tilde{g} (x) | dx &= \int | \int ( g( x ) - g( x - y ) ) \chi(y) dy | dx \\
& \leq \int \chi(y) \int | g(x) - g(x-y) | dx dy \lesssim \beta.
\end{align*}
For the derivative we estimate
\begin{align*}
\int | \tilde{g} \: ' (x) | dx &= \int | \int g (x-y) \chi \: '(y) dy | dx \\
& = \int | \int ( g (x) - g (x-y) ) \chi \: '(y) dy | dx \\
& \leq \int \int | g (x) - g (x-y) | |\chi \:'(y)| dy dx \\
& \lesssim \Delta ^2 |J|^{ -2} \int \int\limits_{ -\Delta^{-1} |J| } ^{ \Delta^{-1} |J| } | g (x) - g (x-y) | dy dx \lesssim \Delta |J|^{-1} \beta.
\end{align*}
To conclude, we observe that $\tilde{g}$ is supported on $2J$ so we can apply \lem{oscillatory} on $2J\cap [-\pi, \pi]$. On $2J \setminus [-\pi,\pi]$ the same bound holds trivially as the phases $\psi$ are linear there.
\begin{align*}
\big| \int_{ \ZR } g (x) e^{ i(P - Q)(x) } dx \big| & \lesssim \beta + \big| \int_{ 2J } \tilde{g} (x) e^{ i(P - Q)(x) } dx \big| \\
& \lesssim \beta + \Big( \sum\limits_{j=n}^m \inf_{2J} \psi_{a_j} \Big)^{-1} \Delta |J|^{-1} \beta \lesssim \beta,
\end{align*}
where in the last line we use \e{derequiv} and $\Delta \gtrsim 1$.
\end{enumerate}
The verification of these conditions implies that $C_0$ can be chosen to be an absolute constant. This finishes the estimate \e{smallscalesestimate}.

\subsection{Improved Local Estimate}\label{improvedlocalestimate}
Let $I_k := [k(1-r), (k+1)(1-r)]$ for $0\leq k \leq \frac{2\pi }{1-r}$ be a partition of the unit circle. Fix some $k$ and denote $I = I_k$. We prove in this subsection that 
\begin{equation}\label{localizedestimate}
    \| T^* \Z1_I g \|_{L^2([-\pi,\pi])} \lesssim R^{\frac{1}{2}} \| g \|_{L^2([-\pi,\pi])},
\end{equation}
where $T^*$ is the adjoint operator of \e{maxoperator10}, i.e.
\begin{equation}
    T^*g(y) = \sum\limits_{s = -1}^{\infty } T^*_s g(y) = T^*_{-1}g(y) + \sum\limits_{s = 0 }^{\infty} \int\limits_{-\pi}^{\pi} g(x) e^{i \theta_{N(x)} (y) } \frac{ \xi_s(x-y)}{\sin \frac{x-y}{2}} dx.
\end{equation}
We know that $\supp \xi_s \subset [-2^{-(s+1)}, - 2^{-(s+3)} ] \cup [2^{-(s+3)}, 2^{-(s+1)} ]$. Let $c(I)$ denote the midpoint of interval $I$. Recall that $|I| = 1-r$, so for $s \leq R-5$ we have
$$\supp T^*_s \Z1_I g \subset B\big( c(I), 2^{-s} \big) \setminus  B\big( c(I), 2^{-(s+4)} \big). $$
Then, the key observation is that  for $| s - s'| \geq 4$ and $s, s' \leq R-5$
\begin{equation}
    T^*_s \Z1_I g(y) T^*_{s'} \Z1_I g(y) = 0.
\end{equation}
Using the estimates for small scales from the previous subsection and H\"older inequality we write
\begin{align*}
    \int_{-\pi}^{\pi} | T^* \Z1_I g (y) |^{2} dy \lesssim \sum\limits_{j = 0}^2
    \sum\limits_{s = 0}^{R-5} \int_{-\pi}^{\pi} | T^*_{\text{Small}} \Z1_I g (y) |^{j} |T_s \Z1_I g(y) |^{2 - j} dy \\
    \lesssim \sum\limits_{j=0}^2 \sum\limits_{s = 0}^{R-5} \int_{-\pi}^{\pi} | T^*_{\text{Small}} \Z1_I g (y) |^{j} |T_s \Z1_I g(y) |^{2 - j} dy \\
    \lesssim \sum\limits_{j=0}^2 \sum\limits_{s = 0}^{R-5} \| T^*_{\text{Small} } \Z1_I g \|_{L^{2} ( [-\pi, \pi] ) }^{j} \| M \Z1_I g \|_{L^{2} ( [-\pi, \pi] ) }^{2 - j} \\
    \lesssim R \| \Z1_I g \|_{L^{2} ( [-\pi, \pi] ) }^{2}.
\end{align*}

\subsection{All scales: $p \geq 2$}\label{allscalespgeq2}
First, let $p=2$. For $g\in L^2([-\pi, \pi])$ we want to prove
\begin{equation}\label{allscalespequals2}
    \int_{-\pi}^{\pi} |T^* g(y)|^2dy \lesssim R \int_{-\pi}^{\pi} |g(y)|^2.
\end{equation}
It suffices to consider $\supp g \subset \cup_{k \text{ even }} I_k$ as the odd case is analogous and the general case follows by a triangle inequality. Also we assume $g$ is real valued as the general case follows by yet another triangle inequality.

We decompose the left-hand side of \e{allscalespequals2} to diagonal and off-diagonal terms as follows.
\begin{align*}
    \int_{-\pi}^{\pi} |T^* g(y)|^2 =& \sum\limits_{k,k' \text{ even}} \int_{-\pi}^{\pi} T^* \Z1_{I_k}g(y) \overline{ T^* \Z1_{I_{k'}}g(y)} \\
    =& \sum\limits_{k\neq k' \text{, even}} \int_{-\pi}^{\pi} T^* \Z1_{I_k}g(y) \overline{ T^* \Z1_{I_k'}g(y)} + \sum\limits_{k\text{ even}} \int_{-\pi}^{\pi} |T^* \Z1_{I_k}g(y)|^2 \\
    =&: \Sigma_{\text{off-diagonal}} + \Sigma_{\text{diagonal}}.
\end{align*}
For the diagonal term we plug in the improved local estimate \e{localizedestimate}.
\begin{align*}
    \Sigma_{\text{diagonal}} & \lesssim R \sum\limits_k \int_{I_k} | g|^2 = R \int_{-\pi}^{\pi} |g(y)|^2.
\end{align*}
We turn to the more interesting off-diagonal sum. Plugging in
\begin{equation*}
    T^* g(y) = \int\limits_{-\pi}^{\pi} g(x) e^{i\theta_{N(x)} (y) } \frac{dx}{\sin \frac{x-y}{2} }
\end{equation*}
we write
\begin{align*}
   |\Sigma_{\text{off-diagonal}} | \leq \sum\limits_{k\neq k'}| \int\limits_{-\pi}^{\pi} \int\limits_{I_k} \int\limits_{I_{k'}} g(x)g(x') e^{ i (\theta _{N(x')} - \theta _{N(x) })(y) } \frac{dxdx'dy}{ \sin \frac{x'-y}{2} \frac{x-y}{2} } | \\
    =\sum\limits_{k\neq k'} | \int\limits_{I_k} \int\limits_{I_{k'}} g(x) g(x') \frac{dxdx'}{ \sin \frac{x-x'}{2} } \int\limits_{-\pi}^{\pi} e^{ i (\theta_{N(x')} - \theta_{N(x) })(y) } \Big( \frac{1}{\tan \frac{x'-y}{2} } - \frac{1}{ \tan \frac{x-y}{2} } \Big) dy |.\numberthis \label{almosthilbertintegration}
\end{align*}
The innermost integral is the circular Hilbert transform of a holomorphic or antiholomorphic function depending on the sign of $N(x') - N(x)$. Thus, we can integrate
\begin{equation}\label{hilbertintegrated}
    \frac{1}{\pi} \int\limits_{-\pi}^{\pi} \frac{ e^{ i (\theta_{N(x')} - \theta_{N(x) })(y) } }{\tan\frac{x'-y}{2}} dy = -i e^{ i (\theta_{N(x')} - \theta_{N(x) })(x') } ( \sgn{ N(x') - N(x) } ).
\end{equation}
To finish the estimate, we plug this into \e{almosthilbertintegration}, take the absolute values inside and note that $|x-x'| \sim (1-r)|k'-k|$.
\begin{align*}
    \e{almosthilbertintegration} \lesssim& \sum\limits_{k\neq k'} | \int\limits_{I_k} \int\limits_{I_{k'}} |g(x)| \cdot |g(x')| \frac{dxdx'}{| \sin \frac{x-x'}{2} | } \\
    \lesssim & \sum\limits_{k\neq k'} \frac{1}{|k'-k|(1-r)} \int\limits_{I_k} |g| \int\limits_{I_{k'}} |g| \lesssim \sum\limits_{k\neq k'} \frac{1}{|k'-k|} \Big( \int\limits_{I_k} |g|^2 \Big)^{\frac{1}{2}} \Big( \int\limits_{I_{k'}} |g|^2 \Big)^{\frac{1}{2}} \\
    =& \sum_{j=1}^{2\pi/(1-r)} \frac{1}{j} \sum_{k} \Big( \int\limits_{I_k} |g|^2 \Big)^{\frac{1}{2}} \Big( \int\limits_{I_{k+j}} |g|^2 \Big)^{\frac{1}{2}} \lesssim R \int\limits_{-\pi}^{\pi} | g|^2.
\end{align*}

We have proved that $\| T^* \|_{ L^2(\ZT) \to L^2(\ZT)} \lesssim R^{\frac{1}{2}}$. Thus, the same bound holds for the operator $T$. As mentioned in the Introduction, sparse Domination theorems 9.1 and 9.2 in  \cite{Kar} for Carleson-type operators with sharp norms directly imply $\| T \|_{L^p( \ZT ) \to L^p( \ZT )} \lesssim  R^{\frac{1}{2}}$ for $p>2$.

\section{Proof of \trm{nontangentialcase}}
Now we assume all points $a_n$ are in the triangle with vertices $(1,0), (1/2,1/2)$ and $(1/2, -1/2)$ and we still want to prove an $L^2(\ZT)$ bound for the operator \e{maxoperator10}. We further assume that $|a_n| \leq r$ for some $0<r<1$ and prove the bounds independent of $r$, then a limiting argument ensures that the same bound holds without this restriction.

Recall the notation of \lem{decomp}, namely, $I_j = [t_j, t_{j+1}]$ where $t_0 = 0, t_j = 2^j (1-r)$ for $1\leq j <N$ and $t_N = \pi$ with $N = [ \log_2 1/(1-r) ]$. In addition, put $J_j := I_j \cup (-I_j)$ and decompose the operator as follows
\begin{align*}
\| Tf \|_{L^2 (\ZT) }^2 &= \sum_{m=0}^N \| \Z1_{J_m} Tf \|_2^2 \\
&= \sum_{m=0}^N \sum\limits_{j,j' = 0}^N \int_{-\pi}^{\pi} \big( \Z1_{J_m} T \Z1_{J_j} f \big) \cdot \big( \Z1_{J_m} T \Z1_{J_{j'}} f \big) \\
& \leq \sum\limits_{j,j' = 0}^N \sum\limits_{m = 0}^N \| \Z1_{J_m} T \Z1_{J_j} f \|_2 \| \Z1_{J_m} T \Z1_{ J_{j'} } f \|_2. \numberthis \label{mainineq}
\end{align*}
Assume for a moment that for $0 < \alpha < 1/2$ and all $j$ and $m$
\begin{equation}\label{localizedineq}
\| \Z1_{J_m} T \Z1_{J_j} f \|_2 \lesssim_{\alpha} 2^{-\alpha |m-j|} \| \Z1_{J_j } f \|_{2}.
\end{equation}
Then, we can continue
\begin{align*}
\e{mainineq} &\lesssim \sum\limits_{j, j' = 0}^{N} \sum\limits_{m = 0}^N 2^{-\alpha |m-j|} \cdot2^{-\alpha |m-j'|} \| \Z1_{J_j } f \|_{2} \| \Z1_{J_{j'} } f \|_{2} \\
& \lesssim \sum\limits_{j,j' = 0}^N \frac{ | j - j' | }{ 2^{ \alpha | j - j'| } } \| \Z1_{J_j } f \|_{2} \| \Z1_{J_{j'} } f \|_{2} \\
& \lesssim \sum\limits_{k=0}^{N} \sum\limits_{j=0}^N \frac{k}{2^{\alpha k}} \| \Z1_{J_j } f \|_{2} \| \Z1_{J_{j + k} } f \|_{2} \\
& \lesssim \sum\limits_{k=0}^N \frac{k}{2^{\alpha k} } \| f \|_2^2 \lesssim \| f \|_2^2,
\end{align*}
where in the penultimate line we have used the Cauchy-Schwarz inequality. We conclude that it suffices to prove \e{localizedineq}. Let us further decompose $J$'s into $I$'s, i.e.
\begin{align*}
\| \Z1_{J_m} T \Z1_{J_j} f \|_2 \leq & \| \Z1_{I_m} T \Z1_{I_j} f \|_2 + \| \Z1_{-I_m} T \Z1_{-I_j} f \|_2 \\
& + \| \Z1_{I_m} T \Z1_{-I_j} f \|_2 + \| \Z1_{-I_m} T \Z1_{I_j} f \|_2.
\end{align*}
We will consider only $\| \Z1_{I_m} T \Z1_{I_j} f \|_2$. All other terms are dealt with in exactly the same way. We only remark for the future application of the localized estimates of \trm{Pavel}, that $\dist (I_j, I_m) \leq \dist (I_j, -I_m)$.

At this moment it is apparent that \e{localizedineq}, with $J$'s replaced by $I$'s, will be deduced from the localized estimates of \trm{Pavel}. As in the previous section, we need to make a transition from the circle to the real line. First, we extend the phases linearly outside $I_j$ preserving the derivatives, namely,
\begin{equation*}
\psi^j_n (x) := \left\{ \begin{aligned}
& \Psi_{a_n}(x), \text{ if } x \in I_j \\
& \Psi_{a_n}(t_{j+1}) + \Psi_{a_n} ' (t_{j+1}) (x-t_{j+1}), \text{ if } x \geq t_{j+1} \\
& \Psi_{a_n}(t_{j}) + \Psi_{a_n} ' (t_{j}) (x-t_{j}), \text{ if } x < t_{j}.
\end{aligned}\right.
\end{equation*}
Let $\xi_{s}$ be as in the previous section and define the Cald\'eron-Zygmund kernel $K(x) := \frac{ 1 }{ \sin \frac{x}{2} } \sum_{s = 10}^{\infty } \xi_{s} (x)$. Then, let $T^j \: : L^0(\ZR) \to L^0 (\ZR)$ be the maximally modulated operator associated to $K$ and the phases, i.e.
\begin{equation*}
T^j f(x) := \sup_n | \int_{\ZR} f(y) e^{-i \sum_{m=1}^n \psi^j_m (y) } K(x-y) dy |.
\end{equation*}
Let $f: \ZT \to \mathbb{C}$ and $\tilde{f} : \ZR \to \mathbb{C}$ such that $\tilde{f} (x) = f(e^{ix})$ for $x\in [-\pi, \pi]$, then
\begin{equation*}
T^j \Z1_{I_j} \tilde{f} (x) = T \Z1_{I_j} f (e^{ix}) \text{ for } x\in [-\pi, \pi],
\end{equation*}
and \e{localizedineq} will follow from
\begin{equation}\label{localizedJineq}
    \| \Z1_{I_m} T^j \Z1_{I_j} f \|_2 \lesssim_{\alpha} 2^{-\alpha |m-j|} \| \Z1_{I_j } f \|_{2}.
\end{equation}

We prove that
\begin{enumerate}
\item[a)] if for $T^j$ the conclusion of \trm{Pavel} holds, then \e{localizedJineq} is true;
\item[b)] \trm{Pavel} holds for $T^j$.
\end{enumerate}
Let us begin with $a)$. We will apply either \e{localized1} or \e{localized2} depending on the ratio $|I_j|/|I_m|$.
Assume $| I_m | < | I_j |$ so that also $m < j$, then for $y \in I_j$ using the partition \e{partition} we get
\begin{align*}
M \Z1_{I_m} (y) \geq \frac{ | I_m | }{ |I_m| + |I_{m+1}| +\cdots + |I_{j-1}| } \gtrsim 2^{-|m-j|}.
\end{align*}
Thus, inequality \e{localized1} implies \e{localizedJineq}. The case $|I_m| \geq |I_j|$ is treated similarly.

We turn to $b)$. Denote
\begin{equation*}
    \ZQ := \{ \sum_{j=1}^n \psi_{a_j} \: : \: n \in \mathbb{N} \} \text{ and } C_2 = 6\pi.
\end{equation*}
We must verify conditions $A,B$ and $C$.
Let us start with the analogue of \e{derequiv}. We want to prove for some $b = a_n$ that
\begin{equation}\label{derratio3}
\sup_{I_j} \psi_{b} ' \sim \inf_{I_j} \psi_{b} ',
\end{equation}

\begin{proof}[Proof of \e{derratio3}]
Let $(J^b_{m})$ be the lacunary decomposition of the circle adapted to $b$. By \lem{decomp} it suffices to prove that the number of $J^b_m$'s required to cover $I_j$ is bounded by an absolute constant.

Recall, that $(I_k)_{k=0}^N$ is a lacunary decomposition of $[0,\pi]$ and choose $j_0$ such that $\arg(b) \in I_{j_0}$. Firstly, if $\arg (b) > \pi /10$, then recall that $b$ is in the non-tangential triangle, so
\begin{equation}\label{trivialcaseadaption}
1-|b| \gtrsim 1 \gtrsim (1-r)2^{j_0}.
\end{equation}
Otherwise, assume $\arg (b) \leq \pi/10$, so $t_j \leq \pi/5$. By the law of sines on the triangle with vertices $0, b$ and $1$ we have
\begin{equation}\label{functioninalpha}
    |b| = \frac{\sin \alpha}{\sin (\alpha + \arg (b)) },
\end{equation}
where $\alpha < \pi/4$ is the angle between the lines $\Im z = 0$ and $\frac{\Im z }{\Re b - 1} = \frac{\Re z - 1}{\Re b - 1}$.
The right-hand side of \e{functioninalpha} is increasing in $\alpha \in [0, \pi/4]$ and decreasing in $\arg b \in [t_{j_0} , t_{j_1}]$, hence
\begin{align*}
1 - |b| & \geq 1 - \sin ( \pi / 4 ) \frac{1}{\sin (\pi / 4 + t_{j_0})} \\
& = \frac{ 2\sin (t_{j_0} / 2) \cos(t_{j_0} /2) }{ \sin ( \pi/4 + t_{j_0} ) } 
\geq c t_{j_0} = c(1-r)2^{j_0}. \numberthis \label{adaption}
\end{align*}
The above inequality proves that $I_{j_0}$ can be covered by at most $1/c$ number of $J^b_m$'s. As the arcs $J^b_m$ increase in geometric progression on both sides of $I_{j_0}$, the same is true for $I_j$.
\end{proof}

Conditions A and B are deduced by exactly the same arguments as in the previous subsection using \e{derratio3}.

To verify Condition C, let us take an interval $J$ with $|J|\leq \pi$, $g\in C^1(J)$ and $P-Q = \sum_{j=n}^m \psi_{b_j}$. We split $J$ into $J_1\subset I_j$ and $J_2\subset I_j^c$. First, on $J_2$ the phases are linear by construction so we trivially have
\begin{equation*}
    | \int_{J_2} e^{i \sum_{j=n}^m \psi_{b_j} (y) } g(y)dy| \lesssim \frac{1}{\sum_{j=n}^m \inf_{J} \psi_{b_j}} ( \frac{1}{|J|} \int_J |g| + \int_J |g'|).
\end{equation*}
On the other hand, for $J_1$ we apply \lem{oscillatory}. The inequalities \e{trivialcaseadaption} and \e{adaption} for $j$ instead of $j_0$ guarantee the hypothesis of the \lem{oscillatory} with $J_1$ for $I$ and $2^j (1-r)$ for $1-r$. Thus, we obtain
\begin{equation*}
    | \int_{J_1} e^{i \sum_{j=n}^m \psi_{b_j} (y)} g(y) dy| \lesssim \frac{1}{\sum_{j=n}^m \inf_{J} \psi_{b_j}} ( \frac{1}{|J|} \int_J |g| + \int_J |g'|).
\end{equation*}
Combining the two estimates and continuing with exactly the same arguments as in the previous subsection, we obtain condition C. Thus, the proof of \trm{nontangentialcase} is complete.

\section{Lower Bound for \trm{compactcase}}

First, we prove an asymptotic formula for the M\"obius phases.
\begin{lemma}\label{phaseasymptotics}
For any $0< r<1$ and $1\leq j\leq \frac{1}{1-r}$ we have
\begin{equation}\label{singlephaseaction}
    | \Psi_{ r } (j(1-r)) - \pi + \frac{1}{ j } | \lesssim \frac{1}{j^2} + 1-r.
\end{equation}
\end{lemma}
\begin{proof}
First of all, if $r \leq \frac{1}{2}$, then the lemma is trivially true. Fix $r>\frac{1}{2}$ and $j$ and let for $\xi \in [r,1]$
\begin{equation}
    f(\xi ):= \Psi_{\xi} ( j (1-\xi) ) - 2\arcsin \frac{j}{\sqrt{1+j^2}}.
\end{equation}
We calculate the derivative
\begin{equation}
    f'(\xi) = \frac{2 \sin j(1-\xi) - j(1-\xi)(1+\xi) }{1 + \xi ^2 -2\xi \cos j(1-\xi)}.
\end{equation}
Then, we estimate it as follows
\begin{align*}
    | f'(\xi) | \lesssim \frac{ | 2\sin j(1-\xi) - 2 j(1-\xi) + 2j(1-\xi)^2 | }{(1-\xi )^2 (1 + j^2)} \\
    \lesssim \frac{ j^3(1-\xi)^3 + j (1-\xi)^2 }{(1-\xi)^2 (1+j^2)} \lesssim (1-\xi)j + \frac{1}{j}.
\end{align*}
Thus, by the fundamental theorem of calculus we write
\begin{equation}\label{eliminatingr}
    |f(r)| = | \int_r^1 f'(\xi) d\xi | \lesssim 1-r.
\end{equation}
On the other hand, one has
\begin{equation}\label{asymptoticofarcsine}
    \pi - 2\arcsin \frac{j}{\sqrt{ 1 + j^2}} = 2 \arcsin \frac{1}{\sqrt{1+j^2}} = \frac{2}{\sqrt{1+j^2}} + o\big( \frac{1}{j^2} \big) = \frac{2}{j} + O\big( \frac{1}{j^2} \big) . 
\end{equation}
Combining the estimates \e{eliminatingr} and \e{asymptoticofarcsine} finishes the proof of the Lemma.
\end{proof}

Let us fix $0<r<1$.  We construct a sequence $(a_n)$ with $|a_n| \leq r$, a choice function $N$ and a function $g$ such that
\begin{equation}
    \| T^*g \|_{L^2 ([-\pi, \pi]) }^{2} \sim \log \Big( \frac{1}{1-r} \Big) \| g\|_{L^2 ([-\pi, \pi]) }^{2}.
\end{equation}
Let $M>0$ be such that $\frac{1}{2}< M e^{M} (1-r) \leq 1$. Then, for $r$ close enough to $1$ we have $M > \frac{1}{2}\log \frac{1}{1-r}$. Also denote $J_k := [kM(1-r); (k+1)M(1-r))$ for $0\leq k \leq e^{M}$. Choose
\begin{equation}
    a_k = r e^{i Mk(1-r)} \text{ for } 0\leq k \leq e^M.
\end{equation}
We enumerate the sequence $(a_n)$ starting from $0$ for the simplicity of the notation. Next, we choose as the linearizing function $N(x) = k$ for $x\in J_k$ and $1\leq k \leq e^M$ and $N(x) = 1$ on the rest of the circle. Recalling the notation $\theta_n(x) = \sum_{k=0}^{n} \Psi_{a_j} (x)$, we can now write the linearized maximal operator corresponding to \e{maxoperator2} by
\begin{equation}
T f(x) := \int_{-\pi}^{\pi} f(y) e^{ - i \theta_{N(x)} (y) } \frac{dy}{\sin \frac{x-y}{2} } .
\end{equation}
The kernel $\frac{1}{\sin \frac{x-y}{2}}$ is chosen to arrive at the circular Hilbert transform after several computations. $T^*$ has the following form
\begin{equation}
    T^* g (y) = \sum \limits_{k = 1}^{e^M} e^{i\theta_k (y)} \int_{J_k} g(x) \frac{dx}{\sin \frac{x-y}{2}} + e^{i\theta_0 (y)} \int\limits_{ ( \bigcup\limits_{0\leq k \leq e^M} J_k )^c } g(x) \frac{dx}{\sin \frac{x-y}{2}}.
\end{equation}
We choose the test function $g$. Denote $I_k := [(k+\frac{1}{4})M(1-r), (k+\frac{3}{4})M(1-r)] \subset J_{k}$.
\begin{equation}
    g(x) = \begin{cases}
    1, \text{ if } x\in I_k \text{ for } 1\leq k \leq e^M \text{ and }k \text{ even}, \\
    0, \text{ otherwise.}
    \end{cases}
\end{equation}
Then, we have
\begin{equation}
    \| g\|_{L^2 ([-\pi, \pi]) }^2 = \frac{1}{4} Me^M (1-r) \sim 1
\end{equation}

We start computing the $L^2$ norm.
\begin{align*}
   \int\limits_{-\pi}^{\pi} | T^*g |^2 = \sum\limits_{ \substack{ k \neq k' \\ 1\leq k,k' \leq e^M \\ k,k' \text{even} }} \int\limits_{-\pi}^{\pi} \int\limits_{I_k} \int\limits_{I_{k'}} g(x') g(x) e^{ i(\theta_{k'} - \theta_{k}) (y) } \frac{dxdx'}{\sin \frac{x-y}{2} \sin \frac{x'-y}{2}} dy \\
   + \sum_{ \substack{ 1\leq k\leq e^M \\ k\text{ even} }} \int\limits_{-\pi}^{\pi} | \int_{I_k} g(x) \frac{dx}{\sin \frac{x-y}{2}} |^2 dy\\
   =: \Sigma_{\text{off-diagonal}} + \Sigma_{\text{diagonal}}.
\end{align*}
The diagonal term will be dominated by the off-diagonal one. We have
\begin{align*}
    \Sigma_{\text{diagonal}} = \sum_k \| H g\Z1_{I_k} \|_{L^2 ([-\pi, \pi]) }^2 \lesssim \sum_k \| g\Z1_{I_k} \|_{L^2 ([-\pi, \pi]) }^2 = \| g \|_{L^2 ([-\pi, \pi]) }^2,
\end{align*}
where $H$ denotes here the Hilbert transform on $\ZT$ with the kernel $\frac{1}{\sin \frac{x-y}{2}}$.
For the off-diagonal term we first do a Fubini and integrate the circular Hilbert transform as in \e{hilbertintegrated} in \subsect{allscalespgeq2}.
\begin{align*}
    \int\limits_{-\pi}^{\pi} e^{ i(\theta_{k'} - \theta_{k}) (y) } \frac{dy}{\sin \frac{x-y}{2} \sin \frac{x'-y}{2}} = \frac{1}{\sin \frac{x'-x}{2}} \int\limits_{-\pi}^{\pi} e^{ i(\theta_{k'} - \theta_{k}) (y) }\Big( \frac{1}{\tan \frac{x-y}{2}} - \frac{1}{\tan \frac{x'-y}{2}} \Big) dy \\
    = \frac{-i \sgn (k'-k) }{\sin \frac{x'-x}{2}} \Big( e^{ i(\theta_{k'} - \theta_{k}) (x) } - e^{ i(\theta_{k'} - \theta_{k}) (x') }\Big). \numberthis\label{hilbertintegration}
\end{align*}
 Plugging in \e{hilbertintegration} and the values of $g$, we continue computing the off-diagonal term.
\begin{align*}
    \Sigma_{\text{off-diagonal}} = -i \sum\limits_{ \substack{ k \neq k' \\ 1\leq k, k' \leq e^M \\ k,k'\text{even} }} \int_{I_{k'}} \int_{I_{k}} \frac{\sgn (k'-k) }{\sin \frac{x'-x}{2}} \Big( e^{ i(\theta_{k'} - \theta_{k}) (x) } - e^{ i(\theta_{k'} - \theta_{k}) (x') }\Big) dxdx'\\
    = \sum\limits_{ \substack{ 1 \leq k < k' \leq e^M \\ k,k'\text{even} } } \int_{I_{k'}} \int_{I_{k}} \frac{ 2 }{\sin \frac{x'-x}{2}} \big( \sin(\theta_{k'} - \theta_{k}) (x) - \sin (\theta_{k'} - \theta_{k}) (x') \big) . \numberthis\label{offdiagonalestimate2}
\end{align*}

From the definition \e{phasedefined}, it follows that  $\Psi_b$ is odd with respect to $\arg b$, that is $\Psi_b (y) = - \Psi_b (2 \arg b - y)$. Also we have $\Psi_b (y) = \Psi_{|b|}  (y - \arg b)$. Using the first identity, then the second one we write
\begin{align*}
    ( \theta_{k'}-\theta_k )(y) = \sum_{j = k+1}^{k'} \Psi_{a_j } (y)
    = - \sum_{j = k+1}^{k'} \Psi_{a_j } ( 2jM(1-r) - y )\\
    = - \sum_{j = k+1}^{k'} \Psi_{a_j} \Big( (k+k'+1)M(1-r) - y - (k+k'+1-2j)M(1-r) \Big) \\
    = - \sum_{j = k+1}^{k'} \Psi_{a_{k+k'+1-j} } ( (k+k'+1)M(1-r) - y ) \\
    = - \sum_{j = k+1}^{k'} \Psi_{a_j} ( (k+k'+1)M(1-r) - y ) \\
    = - ( \theta_{k'}-\theta_k )\Big( (k+k'+1)M(1-r) - y \Big).
\end{align*}
So $\theta_{k'} - \theta_k$ is odd with respect to $\frac{k+k'+1}{2}M(1-r)$. Thus, we can continue from \e{offdiagonalestimate2}.
\begin{align*}
    \Sigma_{\text{off-diagonal}} 
    =& \sum\limits_{ \substack{ 1 \leq k < k' \leq e^M \\ k,k'\text{even} } } \int_{I_k} \int_{I_k} \frac{ 2\sin(\theta_{k'} - \theta_k ) (x) }{\sin \frac{ (k'+k+1)M(1-r) - x' - x }{2} } dxdx'\\
    &- \sum\limits_{ \substack{ 1 \leq k < k' \leq e^M \\ k,k'\text{even} }} \int_{I_k} \int_{I_k} \frac{ 2 \sin(\theta_{k'} - \theta_k ) ( (k'+k+1)M(1-r) - x' ) }{\sin \frac{ (k'+k+1)M(1-r) - x' - x }{2} } dxdx' \\
    =& \sum\limits_{ \substack{ 1 \leq k < k' \leq e^M \\ k,k'\text{even} }  } \int_{I_k} \int_{I_k} \frac{ 4 \sin(\theta_{k'} - \theta_k ) (x) }{\sin \frac{ (k'+k+1)M(1-r) - x' - x }{2} } dxdx'. \numberthis\label{offdiagonalestimate3}
\end{align*}
Applying the rotation symmetry $\Psi_b (y) = \Psi_{|b|}  (y - \arg b)$ mentioned earlier, \lem{phaseasymptotics} and \lem{decomp} we obtain for $x\in I_k$
\begin{align*}
    ( \theta_{k'} - \theta _k )(x) &= ( \theta_{k'} - \theta _k ) ( kM(1-r) ) + (( \theta_{k'} - \theta _k )(x) - ( \theta_{k'} - \theta _k ) ( kM(1-r) )) \\
    & = \pi (k'-k) + \sum\limits_{j=1}^{(k'-k)} \frac{1}{Mj} + O\big( \frac{1}{M} \big) + \int\limits_{kM(1-r)}^{x} ( \theta_{k'} - \theta _k )' (t) dt \\
    &= \pi (k'-k) + \frac{\log (k'-k)}{M} + O\big( \frac{1}{M} \big).
\end{align*}
We know that $\log ( k'-k) \leq M $. Hence, for $r$ sufficiently close to $1$, we have $0\leq \frac{\log (k'-k)}{M} + O\big( \frac{1}{M} \big) \leq 1.1$ and the sine from \e{offdiagonalestimate3} is positive and can be bounded from below. Taking into account that $k'-k$ is even when both $k$ and $k'$ are even, we conclude the estimate as follows.
\begin{align*}
    \Sigma_{\text{off-diagonal}} & \gtrsim \sum\limits_{ \substack{ 1 \leq k < k' \leq e^M \\ k,k'\text{even} }} \frac{ M^2(1-r)^2 \sin \big( \frac{\log (k'-k)}{M} + O\big( \frac{1}{M} \big) \big)  }{ (k'-k)M(1-r) } \\
    & \gtrsim (1-r) \sum\limits_{ 1 \leq k < k' \leq e^M/2 } \frac{ \log (k'-k) + O(1) }{ k'-k } \\
    & = O((1-r)e^M M ) + (1-r) \sum\limits_{ 1 \leq k < k' \leq e^M/2 } \frac{ \log (k'-k)}{ k'-k } \\
    & \gtrsim O((1-r)e^M M) + (1-r)e^M M^2 \gtrsim M \| g\|_2^2.
\end{align*}
As $M \gtrsim \log \frac{1}{1-r}$ we have finished the proof. Note, that for our construction $\| g \|_2 \sim 1 \sim \| g\|_p$. Thus, a H\"older inequality extends this example to a lower bound for $1< p <2$.
\begin{equation}
    \| T^* g \|_{L^{p'} ([-\pi, \pi]) } \gtrsim \| T^* g \|_{L^2 ([-\pi, \pi]) } \gtrsim M \| g\|_{L^2 ([-\pi, \pi]) } \sim M \| g\|_{L^{p'} ([-\pi, \pi]) }.
\end{equation}
\appendix
\section{M\"obius Invariance}
Let us denote
\begin{equation}
    m_b (z) := \frac{z - b}{1 - \overline{b}z}
\end{equation}
the M\"{o}bius transform taking $b$ to $0$. Further, denote
\begin{equation}\label{modulatedhilbert}
    S^{(a_n)}_N f (e^{ix}) := \int_{-\pi}^{\pi} f(e^{iy}) \prod_{j=1}^N m_{a_j}^{-1} (e^{iy}) \frac{dy}{ e^{i(x-y) } - 1 }.
\end{equation}
By \e{operatorclosedform}, $S_N$ is the MT partial sum operator associated to the sequence $(a_n)$, up to a multiplication by a unimodular function and a subtraction of a Hilbert transform. So the maximal operator \e{maxoperator2} is given by $T^{(a_n) } f := | \sup_N S^{(a_n) } _N f|$. We will prove the following proposition on the invariance of the operator norms under the M\"obius transform.
\begin{prop}
Let $(a_n)_{n\geq 1}$ and $b$ be points in the unit disk, then
\begin{equation}\label{Mobius}
\| T^{(a_n)} \|_{L^2 (\ZT) \to L^2 ( \ZT ) } = \| T^{ ( m_{-b} (a_n) ) } \|_{L^2 ( \ZT ) \to L^2 (\ZT ) }.
\end{equation}
Furthermore, if $1 < q < p < \infty$, then
\begin{equation}\label{invariancewithloss}
    \| T^{( m_{-b} (a_n) )} \|_{L^p (\ZT) \to L^p(\ZT)} \leq \delta(q,p) \| T^{(a_n)} \|_{L^q (\ZT) \to L^q(\ZT)},
\end{equation}
where $\delta(q,p) > 0$ are some constants that blow up as $q$ and $p$ get closer.
\end{prop}
As mentioned in the Introduction and in Section 4.3, sparse domination allows to pass from the boundedness of $T^{(a_n)}$ for one $p_0$ to the boundedness for all $p\geq p_0$. Thus, \e{invariancewithloss} implies a symmetric qualitative statement: $T^{(a_n)}$ is bounded on $L^p$ for all $p>r$ if and only if $T^{(m_{-b} (a_n))}$ is bounded on $L^p$ for all $p>r$. Ideally, one might expect also the symmetric quantitative result
\begin{equation*}
     \| T^{( m_{-b} (a_n) )} \|_{L^p (\ZT) \to L^p(\ZT)} \sim_p \| T^{(a_n)} \|_{L^p (\ZT) \to L^p(\ZT)},
\end{equation*}
however, we do not know how to prove or disprove it.

We will need two basic identities that we formulate in the following Lemma.
\begin{lemma}
We have
\begin{align}
& m_a \circ m_{b} (z) = \frac{ 1 + a \overline{b} }{ \overline{ 1 + a \overline{b} } } m_{ m_{-b} (a) } (z), \\
& \int_{-\pi}^{\pi} | f( e^{ix} ) | dx = \int_{-\pi}^{\pi} | f \circ m_{b} ( e^{ix} ) | \frac{( 1 - |b|^2) dx}{ | 1 - \overline{b} e^{ix} |^2 }. \label{L^2norm}
\end{align}
\end{lemma}

\begin{proof}
The first identity is checked by a direct computation.
\begin{align*}
m_a (m_b(z)) &= \frac{ \frac{z - b}{1 - \overline{b} z} - a }{1 - \overline{a} \frac{z - b}{1 - \overline{b} z} } = \frac{ z( 1 + a \overline{b} ) - (a+b) }{1 + \overline{a} b -\overline{(a+b)} z} \\
& = \frac{1 + a \overline{b} }{\overline{1 + a \overline{b}} } \cdot \frac{z - \frac{a+b}{1+a \overline{b} } }{ 1 - \overline{\frac{a+b}{1+ \overline{b}a } } z } = \frac{1 + a \overline{b} }{\overline{1 + a \overline{b}} } m_{m_{-b}} (z).
\end{align*}

The second identity follows from a change of variables. We put $e^{ix} = m_{b} (e^{iy})$ and compute
\begin{align*}
(e^{ix})' = i e^{ix} dx = \frac{ (1 - \overline{b} e^{iy} ) + \overline{b} ( e^{iy} - b ) }{ ( 1 - \overline{b} e^{iy} )^2 } i e^{iy} dy = \frac{1 - |b| ^2 }{ ( 1 - \overline{b} e^{iy} )^2 } e^{iy} dy.
\end{align*}
Hence, we get
\begin{equation}\label{changeofvariables1}
dx = \frac{ 1 - |b|^2 }{ ( 1 - \overline{b} e^{iy} ) (e^{iy} - b) } e^{iy} dy.
\end{equation}
The expression in front of $dy$ also equals $\Psi_b'(y)$, so it is positive. Then we can write
\begin{equation}\label{changeofvariables2}
dx = \Big| \frac{ 1 - |b|^2 }{ ( 1 - \overline{b} e^{iy} ) (e^{iy} - b) } \Big| dy = \frac{1 - |b|^2}{ | 1 - \overline{b} e^{iy} | ^2} dy,
\end{equation}
which finishes the proof of the second identity.
\end{proof}

\begin{proof}[Proof of \e{Mobius}]
We will make a change of variables in \e{modulatedhilbert}. Let
\begin{align}
& e^{ix} = m_b ( e^{iu} ) \text{ and } e^{iy} = m_{b} ( e^{iv} ).
\end{align}
Let us perform the following computation for the kernel.
\begin{align*}
    m_b(e^{iu}) - m_b(e^{iv}) & = \frac{e^{iu} - b}{1 - \overline{b} e^{iu} } - \frac{e^{iv} - b}{1 - \overline{b} e^{iv} }\\
    &= \frac{ e^{iu} - \overline{b} e^{i(u+v)} - b + |b|^2 e^{iv} - (e^{iv} - \overline{b} e^{i(u+v)} - b + |b|^2 e^{iu}) }{ (1 - \overline{b} e^{iu}) (1 - \overline{b} e^{iv}) } \\
    &= \frac{e^{iv} ( e^{ i(u-v) } - 1 ) (1 - |b|^2) }{(1 - \overline{b} e^{iu}) (1 - \overline{b} e^{iv})}.
\end{align*}
Then, applying \e{changeofvariables1} and the above calculation we write
\begin{align*}
S^{(a_n)}_N f(e^{ix}) =& \big( S^{(a_n)}_N f \big) \circ m_b (e^{iu} ) \\
=& \int_{-\pi}^{\pi} f\circ m_b ( e^{iv} ) \prod\limits_{j=1}^N ( m_{a_j} \circ m_b ( e^{iv} ) )^{-1} \\
&\cdot \frac{ (1-|b|^2)e^{iv} }{ ( 1 - \overline{b} e^{iv} )(e^{iv} - b) } \frac{ m_b(e^{iv}) dv }{ m_b(e^{iu}) - m_b(e^{iv}) } \\
=& \int_{-\pi}^{\pi} f \circ m_{b} ( e^{iv} ) \prod\limits_{j=1}^N m_{ m_{-b} (a_j) }^{-1} ( e^{iv} ) \frac{1 - \overline{b} e^{iu}}{1 - \overline{b} e^{iv} } \frac{dv}{ e^{ i(u-v) } - 1 } \\
=& (1 - \overline{b} e^{iu}) S^{ (m_{-b} (a_n) ) }_N ( \frac{1}{1 - \overline{b} e^{i (\cdot) } } (f \circ m_b) ) (e^{iu}).
\end{align*}
Taking absolute values and supremum in $N$ on both sides above, we get
\begin{equation}\label{changeofvariableswithT}
    | \big( T^{(a_n)} f \big) \circ m_b (e^{iu} ) | = | 1 - \overline{b} e^{iu} | | T^{ (m_{-b} (a_n) ) } \big( \frac{1}{1 - \overline{b} e^{i (\cdot) } } (f \circ m_b) \big) (e^{iu}) |.
\end{equation}
Then, using the identity \e{L^2norm} twice we write
\begin{align*}
    \int\limits_{-\pi}^{\pi} | T^{(a_n)} f (e^{ix}) |^2 dx & =
    \int\limits_{-\pi}^{\pi} | ( T^{(a_n)} f )\circ m_b (e^{iu} ) |^2 \frac{1- |b|^2}{|1 - \overline{b} e^{iu} |^2 } du \\
    & = \int\limits_{-\pi}^{\pi} | T^{(m_{-b} (a_n) )} \Big( \frac{( 1 - |b|^2 )^{\frac{1}{2}} }{1-\overline{b} e^{i (\cdot ) }} (f \circ m_b ) \Big) (e^{iu} ) |^2 du \\
    & \leq \| T^{(m_{-b} (a_n) )} \|_{2\to 2}^2 \int\limits_{-\pi}^{\pi} |\frac{ 1 - |b|^2 }{|1-\overline{b} e^{iv } |^2 } (f \circ m_b ) (e^{iv} )  |^2 dv \\
    & = \| T^{(m_{-b} (a_n) )} \|_{2\to 2}^2 \| f \|_2^2.
\end{align*}
Hence, $\| T^{(a_n) )} \|_{2\to 2} \leq \| T^{(m_{-b} (a_n) )} \|_{2\to 2} $ and the reverse inequality follows by interchanging the roles of the sequences.
\end{proof}

\begin{proof}[Proof of \e{invariancewithloss}]
Applying \e{changeofvariableswithT} and \e{L^2norm} as in the proof of \e{Mobius} we modify the $L^p$ norm of the operator as follows.
\begin{align*}
    \int\limits_{-\pi}^{\pi} | T^{(a_n)} f (e^{ix} ) |^p dx =  \int\limits_{-\pi}^{\pi} | ( T^{(a_n)} f )\circ m_b (e^{iu} ) |^p \frac{1- |b|^2}{|1 - \overline{b} e^{iu} |^2 } du \\
    =\int\limits_{-\pi}^{\pi} | T^{(m_{-b} (a_n) )} \Big( \frac{1 }{1-\overline{b} e^{i (\cdot ) }} (f \circ m_b ) \Big) (e^{iu} ) |^p \frac{1 - |b|^2}{ | 1 - \overline{b} e^{iu} |^{2 - p} } du.
\end{align*}
Let $w(u) := \frac{1 - |b|^2}{| 1 - \overline{b} e^{iu} |^{2-p} }$ be the above weight. We claim that the Muckenhoupt's $A_p$ characteristic of $w$ is finite and depends only on $p$, i.e.
\begin{equation}\label{boundedAp}
    [w]_{A_p} \lesssim_p 1.
\end{equation}

\begin{proof}[Proof of \e{boundedAp}]
Assume without a loss of generality that $b = r$ is real and $0 < r < 1$. Also, we know that $\sin x \sim x $ for $x\in [0, \frac{\pi}{2} ]$. Hence, we can write
\begin{align*}
    |1 - \overline{b} e^{iu} |^2 &= |1 - r\cos u - ir\sin u| \\
    &= 1 + r^2 - 2r\cos u = (1-r)^2 + 2r(1-\cos u) \\
    & = (1-r)^2 + 4r \sin ^2 \frac{u}{2} \sim (1-r)^2 + u^2.\numberthis\label{weightequiv}
\end{align*}
Consider an arbitrary interval $(a,b) \subset [-\pi, \pi]$. We need to prove
\begin{equation}\label{Apweight}
    \left( \frac{1}{b-a} \int\limits_a^b w(u) du \right) \left( \frac{1}{b-a} \int\limits_a^b w(u)^{ -\frac{1}{p-1} } du \right)^{p-1} \lesssim_p 1.
\end{equation}
We substitute the weight above with \e{weightequiv}.
\begin{equation}\label{Apweight2}
    \e{Apweight} \lesssim
    \frac{1}{(b-a)^p} \left( \int\limits_a^b ( (1-r)^2 + u^2 )^{ \frac{p-2}{2} } du \right) \left( \int\limits_a^b ( (1-r)^2 + u^2 )^{ -\frac{p-2}{2(p-1)} } du \right)^{p-1}.
\end{equation}
As the integrands on the right-hand side above are even functions, we can restrict our attention to $a \geq 0$. Furthermore, if $b \leq 2 (1-r)$, then $(1-r)^2+u^2$ behaves like a constant on $(a,b)$ and we are done. Otherwise, if $b \geq 2 (1-r)$ then $(1-r)$'s can be neglected and we are left with power weights. The following chain of inequalities finishes the proof.
\begin{align*}
    \e{Apweight2} &\lesssim \frac{1}{(b-a)^p} \left( \int\limits_a^b u^{ p-2 } du \right) \left( \int\limits_a^b u^{ -\frac{p-2}{p-1} } du \right)^{p-1} \\
    &\lesssim \frac{1}{(b-a)^p} \frac{ b^{p-1} - a^{p-1} }{p-1} \left( \frac{ b^{1/(p-1)} - a^{ 1/(p-1) } }{ \frac{1}{p-1} } \right)^{p-1} \lesssim (p-1)^{p-2} \pi^p,
\end{align*}
where in the last inequality we have used the mean value theorem and $b < \pi$.
\end{proof}
The proof can be concluded by an application of Theorems 9.1, 9.2 and 6.4 from  \cite{Kar}. The first two theorems establish the Sparse domination for maximally modulated singular integrals. They claim that there exists a sparse family $\mathcal{S}$ such that
\begin{equation*}
    | T^{(m_{-b} (a_n) )} g (e^{iu} ) | \lesssim \| T^{ (m_{-b} (a_n) ) } \|_{q\to q} \cdot A_{\mathcal{S}, q } g (e^{iu} ),
\end{equation*}
where $A_{\mathcal{S} , q}$ is the corresponding sparse operator with $q$-averages. Theorem 6.4 establishes the boundedness of the sparse operator on $L^{p}(w)$, namely,
\begin{equation*}
    \int\limits_{-\pi}^{\pi} | A_{\mathcal{S},q} g (x)|^p w(x)dx \lesssim \delta(p,q) [w]_{A_p}^{\max (p, p' )} \int\limits_{-\pi}^{\pi} |g (x)|^p w(x)dx.
\end{equation*}
Combining the two estimates with \e{boundedAp} we have
\begin{align*}
    \int\limits_{-\pi}^{\pi} | T^{(a_n)} f (e^{ix} ) |^p dx = \int\limits_{-\pi}^{\pi} | T^{(m_{-b} (a_n) )} \Big( \frac{1 }{1-\overline{b} e^{i (\cdot ) }} (f \circ m_b ) \Big) (e^{iu} ) |^p w(u) du \\
    \lesssim \| T^{( m_{-b} (a_n) ) } \|_{q\to q} \delta(p,q) \int\limits_{-\pi}^{\pi} \frac{1 }{ | 1 -\overline{b} e^{i u } |^p } |(f \circ m_b ) (e^{iu} )|^p w(u) du \\
    = \| T^{( m_{-b} (a_n) ) } \|_{q\to q} \delta(p,q) \int\limits_{-\pi}^{\pi} | f (e^{ix} ) |^p dx,
\end{align*}
where in the last line we have used once again the identity \e{L^2norm}.
\end{proof}

\bibliographystyle{amsalpha}

\bibliography{references}

\end{document}